\documentclass[11pt]{amsart}
\usepackage{amsfonts}
\usepackage{cyrillic}
\newcommand{\cyrrm}{\fontencoding{OT2}\selectfont\textcyrup}
\usepackage[all,cmtip]{xy}
\usepackage{amsmath}
\usepackage{amsthm}
\usepackage{amssymb}
\newtheorem{thm}{Theorem}[section]
\newtheorem{conj}{Conjecture}[section]

\newtheorem*{definition*}         {Definition}
\newtheorem{lemma}[thm]{Lemma}

\newtheorem{cor}[thm]{Corollary}

\newtheorem*{remark}{Remark}
\theoremstyle{remark}
\newcommand*{\Oo}{\mathbb{O}}
\newcommand*{\Q}{\mathbb{Q}}
\newcommand*{\Hh}{\mathbb{H}}

\newcommand*{\Z}{\mathbb{Z}}
\newcommand*{\G}{\mathbb{G}}
\newcommand*{\A}{\mathbb{A}}
\newcommand*{\R}{\mathbb{R}}

\newcommand*{\C}{\mathbb{C}}
\newcommand*{\F}{\mathbb{F}}
\newcommand*{\GQ}{\textrm{Gal}(\overline{\mathbb{\Q}}/\mathbb{Q})}
\newcommand*{\Disc}{\textrm{Disc}}
\newcommand*{\Gal}{\textrm{Gal}}

\author{Jacob Tsimerman}
\begin{document}
\title[Brauer-Siegel for Tori and Lower bounds for CM Galois orbits]{Brauer-Siegel for Arithmetic Tori and lower bounds for Galois orbits of special points}
\maketitle
\begin{abstract}
In \cite{S}, Shyr derived an analogue of Dirichlet's class number formula for arithmetic Tori. 
We use this formula to derive a Brauer-Siegel formula for Tori, relating the Discriminant of a torus to the product of its regulator and class number.
We apply this formula to derive asymptotics and lower bounds for Galois orbits of CM points in the Siegel modular variety $A_{g,1}$. Specifically, we ask that the 
sizes of these orbits grows like a power of Discriminant of the underlying endomorphism algebra. We prove this unconditionally in the case $g\leq 6$, and for all 
$g$ under the Generalized Riemann Hypothesis for CM fields. Along the way we derive a general transfer principle for torsion in ideal class groups of number
fields.
\end{abstract}

\section{Introduction}

Let $A_{g,1}$ be the coarse moduli space of principally polarized Abelian varieties. Recall that an Abelian variety $B$ of dimension $g$ is said to be CM if 
its endomorphism algebra $End(B)\otimes_{\Z}\Q$ contains a semi-simple, commutative algebra $R$ over $\Q$ with $[R:\Q]=2g$.
This paper is motivated by the following conjecture, first suggested by Edixhoven in \cite{EMO}:
	
\begin{conj}\label{lowerboundscm}

Fix an integer $g$. Let $B$ be a $g$-dimensional CM principally polarized Abelian variety, and let $x$ be the corresponding point in 
$A_{g,1}(\overline{\Q})$. Finally, let $Z(End(B))$ denote the center of the 
endomorphism ring of $B$. Then there exists a constant $\delta_g>0$ so that $$|\GQ\cdot x|\gg_g \Disc(Z(End(B)))^{\delta_g}.$$
\end{conj}

	Edixhoven established the desired lower bound for hilbert modular surfaces in \cite{Ed}.
	
	Recently, Pila\cite{P} gave an unconditional proof of the Andre-Oort conjecture for an arbitrary product of modular curves.
A positive answer to question \ref{lowerboundscm} would be one of the main ingredients in generalizing Pila's recent 
work to arbitrary Shimura varieties. Our main theorem is as follows:

\begin{thm}\label{main}
For $g\leq 6$, conjecture \ref{lowerboundscm} holds. If one assumes the Generalized Riemann Hypothesis for CM fields, then conjecture \ref{lowerboundscm} holds for all $g\in\mathbb{N}$. 
\end{thm}

We mention that using similar methods, Ullmo and Yafev independently settled the case of $g=2,3$ \cite{UY2}.
Following Chai and Oort, we define a Weyl CM point to be a CM point $x$ corresponding to an abelian variety $B$ whose endomorphism ring $End(B)$ is an order
in a Weyl CM field. Since for $g=1$ all CM points are Weyl CM points, the following theorem  can be seen as a generalization of Brauer-Siegel's theorem
to all $g$:

\begin{thm}\label{Weyl}
Conjecture \ref{lowerboundscm} holds for all $g\in\mathbb{N}$ if one restricts to Weyl CM points $x$. Moreover, suppose $x$ corresponds to an abelian variety
$B$ whose endomorphism ring $End(B)$ is the ring of integers in a Weyl CM field $K$ with totally real subfield $F$. Then 
$$|\GQ\cdot x|=\left(\frac{\Disc K}{\Disc F}\right)^{\frac12 +o_g(1)}.$$
\end{thm}

	For $g=1$, conjecture \ref{lowerboundscm} is answered by the Brauer-Siegel theorem for Class groups of imaginary quadratic fields. 
For higher $g$, however, the question 
isn't just about understanding the sizes of class groups, but more about the sizes of {\it images of morphisms between class groups}. Thus, for higher $g$        
instead of class groups of number fields, it turns out we are forced to deal with class groups of arithmetic tori over $\Q$. To understand the size of the
class groups of Tori, we prove the following theorem in section 4:

\begin{thm}(Brauer-Siegel Theorem for Tori)\newline

Let $T$ be a torus of dimension $d$ with minimal splitting field $K$, and $[K:\Q]=n$. Let $h_T,R_T$ be the class number and regulator of $T$ respectively. Also, let $\rho$ be the associated integral representation of $\GQ$ and $\rho_{\Q}$ the induced rational representation. Let $f_{\rho}$ be the norm of the Artin conductor
corresponding to $\rho_{\Q}$. Then we have 
\begin{align*}
\frac12\ln(h_TR_T)&=\frac12 \ln f_{\rho}+ o_{n,d}(\ln D_K) \\
&=(\frac12 + o_{n,d}(1))\ln f_{\rho}\\
\end{align*}

\end{thm}

Note that all the relevant tori for the study of CM points are compact at infinity, and thus have no regulator. In this case, theorem \ref{BST} tells us precisely
how large the class group is.

The general setting we study in section 5 is the following: We are given tori $T$ and $S$ over $\Q$, and a map $$\phi:T\rightarrow S$$ between them. 
This induces a map of class groups
$$\tilde{\phi}:Cl_T\rightarrow CL_S$$ and our goal is to understand the size of the image of $\tilde{\phi}$, or more precisely to get a lower bound for it.
One way to do this is to try and bound torsion in class groups of Tori: suppose we knew that $\phi$ had an inverse up to isogeny, so that there is 
some map $\psi:S\rightarrow T$ with $\phi\circ\psi$ is the $\times n$ map on $T$. Then it follows from this that the image of 
$\tilde{\phi}$ is at least the image of $\times n$ map on $Cl_S$, and so it has size at least $|Cl_S|/|Cl_S[n]|$. Now if we could just give a good bound for 	
$|Cl_S[n]|$ we would be done. Unfortunately, bounding torsion in class groups is a very hard problem in general. We take a more precise approach in section 5,
which describes the size of the cokernel of $\tilde{\phi}$ in terms of the cokernel of the map on characters $X(S)\rightarrow X(T)$ viewed as a Galois
representation. A nice corollary of this approach is that it allows us to derive new transfer principles between torsion in class groups of number fields, as is
explained in section 6. 
	
	The structure of this paper is as follows. In sections 2 and 3 we give background about arithmetic tori, and recall Shyr's formula,
which is the analogue for arithmetic Tori of Dirichlet's class number formula for number fields. In section 4 we use Shyr's formula to prove the Brauer-Siegel
Theorem for tori \ref{BST}. Section 5 is the core of the paper, and it proves a structure theorem for maps between class groups of Tori. In section 6 we 
explain how one can use the results in Section 5 to derive transfer principles for torsion in ideal class groups of number fields, 
and give a new application relating 2-torsion
in quartic fields to 2-torsion in their cubic resolvents. Finally, section 7 gives the promised applications to lower bounds of Galois orbits of CM abelian 
varieties. 

\section{Preliminaries}

A Torus of dimension $d$ over a field $k$ is an algebraic group $T$ over $k$ such that $T_{\bar{k}}$ 
is isomorphic to a power of the multiplicative algebraic group $\G_m$,so that  $T_{\bar{k}}\approx {\G_m^d}_{/\bar{k}}$. 

Given a torus $T$ of dimension $d$, we define its character group $X(T)$ to be $Hom_{\bar{k}}(T,\G_m)$. 
Note that $X(T)$ is a free abelian group on $d$ generators.
$X(T)$ also acquires an action of the absolute Galois group $G_k:=\textrm{Gal}(\bar{k}/k)$ via $\chi^{\sigma}(t)=\chi(t^{\sigma^{-1}})^{\sigma}$.
The functor sending $T$ to $X(T)$ is a contravariant category equivalence between Tori and torsion-free,finitely generated $G_k$-modules.

We say that a torus $T$ splits over $K$ if it is isomorphic to a power of $\G_m$ over $K$, or 
equivalently that $X(T)$ is a trivial module for the subgroup $\Gal(\bar{k}/K)$. 
We define $C(K/k)$ to be the set of all Tori over $k$ which are split over $K$. Note that 
there is a minimal extension $K$ of $k$ over which $T$ becomes split, and $K$ is Galois over $k$.
For each extension $L$ of $k$, we define $X(T)_L$ to be the subgroup of $X(T)$ fixed by $\Gal(\bar{k}/(L\cap\bar{k}))$. If $v$ is a place of $k$, we sometimes
write $X(T)_v$ for $X(T)_{k_v}$.

We shall require the following basic statement from the theory of integral representations of finite groups:

\begin{thm}\label{basicequality}

Let $G$ be a finite group and fix a positive integer $d$. Then there are only $O_{d}(1)$ integral representations of $G$ of dimensions $d$ up to isomorphism.

\end{thm}

\begin{proof}

Suppose $L$ is a $d$-dimensional integral representation of $G$. We wish to 
find a $G$-invariant positive-definite inner product on $L$ such that the non-zero vectors in $L$ of smallest norm generate $L$.
To show this is possible, fix a $G$-invariant positive-definite inner product $\langle,\rangle_1$ on $L$, and consider the span $M\subset L$ of the vectors
of smallest norm. If $M$ is not all of $L$, then we can consider the dual $N$ of $M$ with respect to $\langle,\rangle_1$. Both $M$ and $N$ are $G$-invariant, so
we can find a $G$-invariant inner product $\langle,\rangle_2$ which is trivial on $M$ and positive-definite on $N$. By considering 
$x\langle,\rangle_1+\langle,\rangle_2$ for an appropriate $x>0$, we can modify our inner product so as to strictly increase the set of non-zero vectors of
smallest norm and also their span. 

Thus we can find an inner product $\langle,\rangle$ on $L$ which is $G$-invariant and a basis $v_1,\dots,v_n$ consisting of 
elements of $L$, such that 
\begin{equation}\label{minnorm}
\forall i, \langle v_i,v_i\rangle=y
\end{equation}
 for some $y>0$ and the norm of any other non-zero vector in $L$
is at least $y$. Let $\tilde{L}$ be the span of the $v_i$ (note that $\tilde{L}$ might not be $G$-invariant). 

By \eqref{minnorm} the sphere of radius $y/2$ is contained in a fundamental domain for $L$, and the covolume of $\tilde{L}$ is at most $y^d$. We thus have that
the index $[L:\tilde{L}]$, which is equal to the ratio of the covolumes, is at most $\frac{2^d}{Vol(S^d)}$. Setting 
$N=(2^d)!$ we see that $N\cdot L\subset\tilde{L}$. Now consider the matrix representation for an element $g\in G$ w.r.t the basis for $v_1,\cdots v_n$.

 Set $gv_j=\sum_i a_{ij}v_i$. Since $N\cdot L\subset\tilde{L}$ we have $a_{ij}\in\Z[\frac{1}{N}]$. Since the norm of $gv$ is also $y$, by considering the basis 
$v_1,\dots,v_{i-1},gv,v_{i+1},\dots,v_n$ and considering covolumes we conclude as above that $|a_{ij}|\leq N$. Thus there are only $O_d(1)$ options for the
representation of $G$ on $\tilde{L}$. Since there are only $O_d(1)$ options for $L$ given $\tilde{L}$, this completes the proof.

\end{proof}

In what follows, we write $\A$ for the Adele group, and $\A_f$ for the finite Adeles. likewise, for a number field $K$ we write $\A^K$ and $\A^K_f$ for the 
Adeles (resp. finite Adeles) over $K$.

\section{Invariants of Arithmetic Tori}

In this section we define several invariants of arithmetic tori and state Shyrs analogue of the class number formula. 
From now on we restrict to tori defined over $\Q$.
A useful example to keep in mind is the following: Given a number field $L$, the multiplicative group $L^{\times}$ can be thought of as a torus over $\Q$
given by restricting the multiplicative group $\G_m$ from $L$ to $\Q$. The corresponding character group $X(T)$ is the free $\Z$-module generated by the
complex embeddings $\psi_i:L\rightarrow\C$ with the natural Galois action. We denote this torus by $Res_{L/\Q}\G_m$.

\subsection{Class number $h_T$:}$\newline$

Define $T(\Z_p)$ to be the maximal compact subgroup of $T(\Q_p)$, and define $T(\R)^c$ to be the maximal compact subgroup of $T(\R)$. It follows 
that $$T(\hat{\Z}):=\prod_{p}T(\Z_p)$$ is the maximal compact subgroup of $T(\A_f)$, and $T(\A)^c:=T(\hat{\Z})\times T(\R)^c$ 
is the maximal compact subgroup of $T(\A)$.

 Define the class group of $T$ to be $$CL_T:=T(\Q)\backslash T(\A_f)\slash T(\hat{\Z}).$$  The class number of $T$ is defined to be $$h_T:= |CL_T|.$$ 
 Note that for $T=Res_{L/\Q}\G_m$, the class number $h_T$ is the ordinary class number of $L$.
 
\subsection{Units $w_T$ and Regulator $R_T$:}$\newline$

Consider a basis of $X(T)_{\infty}$ given by $\xi_1,\dots,\xi_{r_\infty}$ such that $\xi_1,\dots,\xi_r$ is a basis of $X(T)_{\Q}$. That such a basis can be chosen
follows from the fact that $X(T)_{\Q}$ is a primitive sublattice of $X(T)$, being the fixed set of a group action.
Now, the unit group $T(\Q)\cap T(\hat{\Z})$ has the decomposition $$T(\Q)^{tor}\times E,$$ where $E$ is a free abelian group of rank $r_{\infty}-r$ 
and $T(\Q)^{tor}=T(\Q)\cap T(\A)^c$ is finite, since its both compact and discrete in $T(\A)$. Pick a basis $e_i$ for $E$.
 Define the number $w_T$ and the regulator 
$R_T$ as follows:

$$w_T = \left|T(\Q)\cap T(\A)^c\right|,$$ and $$R_T = |\det(\log|\xi_{r+i}(e_j)|)|,1\leq i,j\leq r_{\infty}-r.$$

Note that for $T=Res_{L/\Q}\G_m$, $w_T$ is the number of roots of unity in $L$, while $R_T$ is the regulator of $L$.

\subsection{Quasi-residue $\rho_T$:}$\newline$

Let $\rho:\GQ\rightarrow Gl_d(\Z)$ be the representation corresponding to $T$, and let $\rho_{\Q}$ be the induced rational representation. We can then consider
the Artin L-function $L(s,\rho_{\Q})$. Recall that $r$ is the rank of $X(T)_{\Q}$ and define the quasi-residue $\rho_T$ of $T$ to be

$$\rho_T:= \lim_{s\rightarrow 1} (s-1)^r L(s,\rho_{\Q}).$$

Note that for $T=Res_{L/\Q}\G_m$, $\rho_T$ is the residue at $s=1$ of the Dedekind-Zeta function $\zeta_L(s)$.

\subsection{Tamagawa number $\tau_T$:}$\newline$

Fix an invariant gauge form $\omega$ on $T$ defined over $\Q$. Then $\omega$ induces canonically a Haar measure $\omega_v$ on $T(\Q_v)$ for each $v$. 
Moreover, for almost all finite primes $p$ we have $L_p(1,\rho_{\Q})\omega_p(T(\Z_p))=1$, and so $\omega_{\infty}\prod_{p}L_p(1,\rho_{\Q})\omega_p(T(\Z_p))$ defines a Haar measure
$\omega_T$ on $T(\A)$, which is moreover independent of the original choice of $\omega$ by the product formula.
	
Define $T(\A)^1$ to be the kernel of the map $\Lambda: T(\A)\rightarrow (\R_{+}^{\times})^r$ given by $$\Lambda(x)=(||\xi_i(x)||_{\Q})_{1\leq i\leq r}$$ where
$||\cdot||_{\Q}$ denotes the Idelle volume of $\Q$. Pulling back the measure $\Lambda_{i=1}^r t_i^{-1}dt_i$ gives us a measure $d\tilde{t}$ on $T(\A)/T(\A)^1$. 
Let $m$ be the measure on $T(\A)^1$ such that $m$ and $d\tilde{t}$ glue together to give $\rho_T^{-1}\omega_T$. We define the Tamagawa number $\tau_T$ to be
the measure of $T(\A)^1/T(\Q)$ under $m$. 

For $T=Res_{L/\Q}\G_m$, the Tamagawa number is $1$. 

\subsection{Quasi-Discriminant $D_T$:}$\newline$

Assigning $T(\R)^c$ measure 1 and gluing it with the measure $\Lambda_{i=1}^{r_{\infty}} d\xi_i$ on $T(\R)/T(\R)^c$ we get a measure on $T(\R)$ which we 
also denote by $d\tilde{t}$. Now, define $$c_T = \frac{w_{\infty}}{d\tilde{t}}\prod_{p}L_p(1,\rho_{\Q})\omega_p(T(\Z_p)).$$ The quasi-discriminant $D_T$ is defined to 
be $D_T=1/c_T^2$.

For $T=Res_{L/\Q}\G_m$ the quasi-discriminant $D_T$ is the ordinary discriminant $D_L$.

\subsection{Analytic Class number Formula for Tori:}
$\newline$

We have the following result of Shyr \cite{S}:

\begin{thm}(Shyr)\label{shyr}

The following equality holds: $$h_T\cdot R_T=\rho_T \cdot w_T\cdot\tau_T\cdot|D_T|^{1/2}.$$ 

\end{thm}

For $T=Res_{L/\Q}\G_m$, this reduces to Dirichlet's class number formula for $L$.

\section{Brauer-Siegel formula for Tori}
Define $K$ to be the splitting field of $T$, and let $D_K$ be its discriminant. Set $d$ to be the dimension of $T$ and $n$ to be $[K:\Q]$.
In this section we use Shyr's class number formula to derive a Brauer-Siegel theorem for Tori by estimating the sizes of the invariants $w_T,\rho_T,\tau_T$ and 
$D_T$ with respect to $D_K$. We show that except for $D_T$, the invariants are
all small. To make this precise, we find it useful also to introduce the following notation:  A function $F(T,K)$ is said to be \emph{discriminant negligible} if 
$\forall\epsilon>0$, we have $$D_K^{-\epsilon}\ll_{n,d,\epsilon} |F(K,T)|\ll_{n,d,\epsilon} D_K^{\epsilon}.$$

\begin{lemma}
$\rho_T$ is Discriminant negligible.
\end{lemma}

\begin{remark}

This is the only part of the paper which is ineffective. That is, we cannot write down a $C_{\epsilon,n,d}>0$ such that 
$$\frac{1}{C_{\epsilon,n,d}}\cdot D_K^{-\epsilon} \leq |F(K,T)| \leq C_{\epsilon,n,d}\cdot D_K^{\epsilon}.$$ This stems from Siegel's ineffective lower bound
for $L(1,\chi)$ where $\chi$ is a quadratic dirichlet character, and is the reason why the Bruaer-Siegel theorem \ref{BST} and our lower bounds for Galois
orbits of CM points are all ineffective. 

\end{remark}

\begin{proof}

Let $L$ be a subfield of $K$ such that $\Gal(K/L)$ is Galois, and $\chi$ be a character of $\Gal(L/\Q)$. Then the estimate
$$D_K^{-\epsilon}\ll_{\epsilon} L(\chi,1)\ll_{\epsilon} D_K^{\epsilon}$$ is known by Siegel's work. By Brauer's induction theorem $L(\rho_{\Q},1)$ can be 
expressed as a quotient of products of such L-functions, and the result follows.

\end{proof}

\begin{lemma}

$\tau_T = O_{n,d}(1)$, and $1/\tau_T=O_{n,d}(1)$. Therefore $\tau_T$ is Discriminant negligible.

\end{lemma}

\begin{proof}
The Tate-Shafarevich group ${\cyrrm{SH}}_T$ is defined to be the kernel of the map
$$H^1(\GQ,T(\Q))\rightarrow\prod_p H^1(\Gal(\overline{\Q_p}/\Q_p,T(\Q_p)).$$ A  result of Ono \cite{O} says that $$\tau_T = 
\frac{|{\cyrrm{SH}}_T|}{|H^1(\Gal(K/\Q),X(T))|}.$$ Now, since we are fixing $d$ and $n$ , the group $\Gal(K/\Q)$ is in one of a finite 
number of isomorphism classes of groups and by lemma \ref{basicequality} the pair $(\Gal(K,\Q),X(T))$ goes over a finite set of isomorphism types of $(G,M)$ 
where $G$ is a group and $M$ is a representation of $G$. 

Therefore $|H^1(\Gal(K/\Q),X(T))|$ is $O_{n,d}(1).$ Likewise, as stated in the proof of (\cite{PR},Prop 6.9) the group 
${\cyrrm{SH}}_T$ is a quotient of $H^2(\Gal(K/\Q),X(T))$ and so its order is $O_{n,d}(1)$ by the same argument.

\end{proof}

\begin{lemma}\label{unit}

The order of the group of units satisfies $w_T = O_{n,d}(1)$, and so $w_T$ is Discriminant negligible.

\end{lemma}

Note $T(\Q)^{tor}$ is a subgroup of $T(K)^{tor}\approx ((K^{\times})^{tor})^n$. The degree of a primitive $q'th$ root of unity over $\Q$ is $\phi(q)$. Since 
$\phi(q)$ goes
to infinity with $q$ and the degree of $K$ is $d$, the order of the unit group of $K$ is $O_d(1)$. The result follows.

\begin{lemma} \label{isog}

 Let $T$ and $T'$ be isogenous Tori in $C(K/\Q)$, of dimension $d$. Then $D_T=D_{T'} \cdot O_{\epsilon,d,n}(D_K^{\epsilon})$.
 
 \end{lemma}
 
 \begin{proof}
 
 By Lemma \ref{basicequality} there exists an isogeny $\lambda:T\rightarrow T'$ between $T$ and $T'$ of degree $m$, where $m=O_{d,n}(1)$. For each
 finite prime $p$ this induces a map $\lambda_p^c:T(\Z_p)\rightarrow T'(\Z_p)$. Define 
$q(\lambda_p^c)=\frac{|\textrm{Coker}(\lambda_p^c)|}{|\textrm{Ker}(\lambda_p^c)|}$. Then by the discussion at the beginning of section 4 of \cite{S}, we have
$$D_T/D_{T'} = \prod_{p}q(\lambda_p^c).$$

Now, as Shyr points out at the end of section 2 of \cite{S}, $q(\lambda_p^c)=1$ for all primes $p$ that are unramified in $K$ and satisfy $(p,m)=1$. 
Since $\lambda$ is an isogeny of degree $m$, it factors through the $\times m$ map, and so we have 
$$|\textrm{Coker}(\lambda_p^c)|\leq |T'(\Z_p)[m]|\leq|T'(\overline{\Q_p})[m]|= m^d.$$ Setting $S$ to be the set of all primes ramified in $K$, we thus have
$$D_T=D_{T'}\times O_{d,n}(m^{d|S|})$$ from which the result follows.

 \end{proof}

\begin{cor}\label{discartin}
Let $T$ be a torus in $C(K/\Q)$ of dimension $d$. Let $\rho$ be the associated integral Galois representation, and 
$\rho_{\Q}$ the induced rational representation. Let $f_{\rho}$ be the norm of the Artin conductor corresponding to $\rho_{\Q}$. Then
$\frac{f_{\rho}}{D_T}$ is Discriminant negligible.
\end{cor}

\begin{proof}

For tori of the form $R_{L/\Q}(\G_m)$, where $L$ is a subfield of $K$, we actually have equality since the Artin conductor and the quasi-discriminant
are both equal to the discriminant of $L$. By Theorem 1.5.1 in \cite{O2} there are integers $m,m_i,n_i $ of size $O_{n,d}(1)$ such that $T^m\times 
\prod_{i}(R_{L_i/\Q}(\G_m))^{m_i}$ is isogenous to $\prod_{i}(R_{L_i/\Q}(\G_m))^{n_i}$ where
$L_i$ goes over all subfields of $K$. Since both the quasi-discriminant and the Artin conductor are multiplicative under products, the 
corollary now follows from Lemma \ref{isog}.

\end{proof}

The following lemma is essential in showing that the class number of Tori is large in terms of the \emph{Discriminant} of the splitting field.
 
\begin{lemma}\label{artin}

If $K$ is a Galois extension of $\Q$ of degree $n$ and $\rho:\Gal(K/Q)\rightarrow GL(V)$ is a representation of $\Gal(K/\Q)$ on a complex
vector space $V$ with trivial kernel, then $$f_{\rho}\gg_n D_K^{\frac1{n-1}}.$$

\end{lemma}

\begin{proof}

Let $\rho:\Gal(K/\Q)\rightarrow GL(V)$. Since $\rho$ has 
trivial kernel, its non-trivial on all non-trivial inertia groups. Let $I_p$ denote the inertia group of a prime over $p$. Since $v_p(f_{\rho})\geq 
dim_{\C}(V/V^{I_p})$, each ramified prime $p$ of $K$ divides $f_{\rho}$. Now, each prime which is tamely ramified divides $D_K$ with exponent
at most $n-1$. Since the contribution of the wildly ramified primes is a priori bounded by the dimension, we have $$D_K\ll_n \prod_{p}p^{n-1}$$ where the product 
is over ramified primes in $K$. This implies the lemma.

\end{proof}

Putting together all the above we arrive at 

\begin{thm}(Brauer-Siegel Theorem for Tori)\label{BST}\newline

Let $T$ be a torus of dimension $d$ in $C(K/\Q)$ with $[K:\Q]=n$, with $K$ being the minimal splitting field of $T$. Let $h_T,R_T$ be the class number and regulator of $T$ respectively. Also, let $\rho$ be the associated integral representation of $\GQ$ and $\rho_{\Q}$ the induced rational representation. Let $f_{\rho}$ be the norm of the Artin conductor
corresponding to $\rho_{\Q}$. Then we have 
\begin{align*}
\frac12\ln(h_TR_T)&=\frac12 \ln f_{\rho}+ o_{n,d}(\ln D_K) \\
&=(\frac12 + o_{n,d}(1))\ln f_{\rho}\\
\end{align*}

\end{thm}

\begin{proof}

The first equality is an immediate consequence of the previous lemmas and Theorem \ref{shyr}. The second equality follows from Lemma \ref{artin} and the fact that
$K$ is the minimal splitting field for $T$.
\end{proof}

\section{Structure theorems for Class groups of Tori}

We focus in this section on the following setup: Suppose we have a map $\phi:S\rightarrow T$ of tori defined over $\Q$. This then induces a map of class groups
$\phi_0:Cl_S\rightarrow Cl_T$, and we wish to understand how large the image of $\phi_0$ is. This question can be very difficult. For instance, suppose that
$S\cong T$ and $\phi$ is the map sending $t$ to $t^n$. Then the question reduces to understanding the size of the group $Cl_T[n]$, and bounding torsion
for class groups of number fields is notoriously difficult. We show, however, that this is essentially the only hard case. More precisely, we prove the following theorem:

\begin{thm}\label{torusexact}
Let 
$$\xymatrix{
1\ar[r]&S\ar[r]^{f}&T\ar[r]^g&U\ar[r]&1\\
}$$
be an exact sequence of Tori, where $S,T,U$ are in $C(K/\Q), [K:\Q]=n$ and the dimension of $T$ is $d$. Then the homology groups of the induced sequence
$$\xymatrix{
1\ar[r]&Cl_S\ar[r]^{f_0}&Cl_T\ar[r]^{g_0}&Cl_U\ar[r]&1\\
}$$
are Discriminant negligible.
\end{thm}

\begin{proof}

For a torus $T$ define $Cl(T)$ to be $T(\Q)\backslash T(\A_f)$.
By considering the cohomology sequence associated to the exact sequence 
\begin{equation}\label{Adeleseq}
\xymatrix{
1\ar[r]&T(K)\ar[r]&T(\A^K)\ar[r]&T(K)\backslash T(\A^K)\ar[r]&1\\
} 
\end{equation}

we see that there is an injection $$\psi_T:Cl(T)\hookrightarrow H^0(\Gal(K/\Q),T(K)\backslash T(\A^K)).$$

Moreover, the image contains the image of the norm map on $T(K)\backslash T(\A^K)$ and thus the size of the cokernel of $\psi_T$ is at most the size of the
Tate cohomology group $\hat{H}^0(\Gal(K/\Q),T(K)\backslash T(\A^K))$. By Nakayama-Tate duality this is isomorphic to 
$H^2(\Gal(K/\Q),X(T))$ which is $O_{n,d}(1)$ by Lemma \ref{basicequality} as has already been discussed. 

Now, suppose we have an exact
sequence of Tori $$\xymatrix{
1\ar[r]&S\ar[r]^{f}&T\ar[r]^g&U\ar[r]&1\\
}$$
as in the statement. We first prove that $g_0$ has a small cokernel, and then that $f_0$ has a small kernel. The theorem then follows since 
$h_T(h_Sh_U)^{-1}$ is Discriminant negligible by lemma \ref{BST}.

\begin{itemize}

\item \emph{$g_0$ has small cokernel:}

Since all our tori become split over $K$, we have an exact sequence

$$\xymatrix{
1\ar[r]&S(K)\backslash S(\A^K)\ar[r]^{f_1}&T(K)\backslash T(\A^K)\ar[r]^{g_1}&U(K)\backslash U(\A^K)\ar[r]&1\\
}.$$

Taking Cohomology, we get an exact sequence 
$$H^0(\Gal(K/\Q),T(K)\backslash T(\A^K))\rightarrow H^0(\Gal(K/\Q),U(K)\backslash U(\A^K))\rightarrow H^1(\Gal(K/\Q),S(K)\backslash S(\A^K)).$$

By Nakayama Tate duality, $$H^1(\Gal(K/\Q),S(K)\backslash S(\A^K))\cong H^1(\Gal(K/\Q),X(S))$$ and is thus of size $O_{n,d}(1)$. Consider the commutative diagram
$$\xymatrix{
Cl(T)\ar[r]\ar[d]^{\psi_T}&Cl(U)\ar[d]^{\psi_U}\\
H^0(\Gal(K/\Q),T(K)\backslash T(\A^K))\ar[r]&H^0(\Gal(K/\Q),U(K)\backslash U(\A^K))\\
}.$$

The downward maps are injective with cokernels of size $O_{n,d}(1)$ and the bottom map has cokernel of size $O_{n,d}(1)$, so the map $Cl(T)\rightarrow Cl(U)$ 
must also have cokernel of size $O_{n,d}(1)$. Finally, consider the commutative diagram
$$\xymatrix{
Cl(T)\ar[r]\ar[d]&Cl(U)\ar[d]\\
Cl_T\ar[r]^{g_0}&Cl_U\\
}.$$

The top map has cokernel of size $O_{n,d}(1)$ and the downwards maps are surjective, so we see that $g_0$ also has cokernel of size $O_{n,d}(1)$.

\item\emph{$f_0$ has small kernel:}

For a ring $R$, we consider $S(R)$ as a subgroup of $T(R)$. We have

$$\textrm{ker}(f_0) \approx \frac{T(\hat{\Z})T(\Q)\cap S(\A_f)}{S(\hat{\Z})S(\Q)}.$$

By group theory we have: \[\left|\frac{T(\hat{\Z})T(\Q)\cap S(\A_f)}{S(\hat{\Z})S(\Q)}\right|\leq\left|\frac{T(\hat{\Z})\cap S(\A_f)}{S(\hat{\Z})}\right|\cdot
\left|\frac{T(\hat{\Z})S(\A_f)\cap T(\Q)}{S(\Q)(T(\hat{\Z})\cap T(\Q))}\right|\]

The first term on the right hand side is 1 since $S(\hat{\Z})$ is the maximal compact subgroup of $S(\A_f)$. For the second part, consider the map $g$ from 
$T(\A_f)$ to $U(\A_f)$. Since the kernel of $g$ is $S(\A_f)$ and $$T(\Q)\cap S(\A_f)= S(\Q),$$ this induces an injective map $\tilde{g}$ on 
\[\frac{T(\hat{\Z})S(\A_f)\cap T(\Q)}{S(\Q)(T(\hat{\Z})\cap T(\Q))}.\]

The image of $g$ on $T(\hat{\Z})S(\A_f)\cap T(\Q)$ lies inside $U(\Q)\cap U(\hat{\Z}).$ Now, $$U(\Q)\cap U(\hat{\Z})=U(\Q)^{tor}\times E_U$$ where $E_U$ is a finitely generated free abelian group of rank at most $d$.
Also, corresponding to the map $g:T\rightarrow U$ there is a map $h:U\rightarrow T$ such that $g\circ h$ is the multiplication by $m$ map on $U$, where
$m=O_{n,d}(1)$. This implies that the image of $g$ on $S(\Q)(T(\hat{\Z})\cap T(\Q)$ is at least the image of the multiplication by $m$ map on 
$E_U$. Therefore, the size of the image of $\tilde{g}$ is bounded by $|E_U[m]|\cdot|U(\Q)^{tor}|\leq m^d\cdot w_U$, which means $ker(f_0)=O_{d}(1)$ as desired.

\end{itemize}

\end{proof}





As a byproduct of the above, we get the following important corollary, which can be considered as a very general transfer principle for torsion in class groups.

\begin{cor}\label{transfer}

Let $M$ be a fixed finite abelian group on which $\Gal(K/\Q)$ acts, where $[K:\Q]$ is bounded by $n$. 
Let $T_1,T_2,S_1,S_2$ be Tori which split over $K$ whose dimensions are bounded by $d$. Let $\phi_i:T_i\rightarrow S_i, i=1,2$ be maps  which are surjective
as maps of Tori, and such that the cokernel of $X(T_i)^*\rightarrow X(S_i)^*$ is isomorphic to $M$ as a $\Gal(K/\Q)$ module. Finally, let $C_i$ be the size
of the cokernel of the induced map $\widetilde{\phi_i}:Cl_{T_i}\rightarrow Cl_{S_i}$. Then $\frac{C_1}{C_2}$ is discriminant negligible. 

\end{cor}

\begin{proof}

We shall work with cocharacters for the purpose of this proof. We have maps $\psi_i:X(S_i)^*\rightarrow M$ whose kernels
are $\phi_i(X(T_i)^*)$. Let $$L_S\subset X(S_1)^*\oplus X(S_2)^*$$ be the kernel of the map $$\Phi:X(S_1)^*\oplus X(S_2)^*\rightarrow M$$ defined by
$\Phi(a,b) = \psi_1(a)+\psi_2(b)$, and let $U$ be the torus whose cocharacter module is $L_S$. We then have the following commutative diagram:
$$\xymatrix{
0\ar[r]&X(T_1)^*\ar[r]^{\phi_1}&X(S_1)^*\ar[r]^{\psi_1}&M\ar[r]&0\\
0\ar[r]&X(T_1\oplus T_2)^*\ar[r]\ar[u]& L_S\ar[r]^{\psi_1}\ar[u]&M\ar[u]\ar[r]&0\\
0\ar[r]&X(T_2)^*\ar[r]^{\phi_2}\ar[u]&\phi_2(X(T_2)^*)\ar[u]\ar[r]^{\;\;\;\;\;\psi_2}&0\ar[u]\\
}.$$

The vertical and horizontal sequences are exact. Now consider the induced diagram on Class groups. Since $\phi_2$ is surjective, it follows from
theorem \ref{torusexact} that the induced map on class groups is surjective up to discriminant negligible factors. Also, the induced 
vertical sequences are exact up to discriminant negligible factors. Let $C$ be the size of the cokernel of $Cl_{T_1}\oplus Cl_{T_2}\rightarrow Cl_U$. 
An application of the snake lemma shows that $\frac{C_1}{C}$ is discriminant negligible. Likewise, $\frac{C_2}{C}$ is discriminant negligible. The result follows.

\end{proof}

We denote the size of the above cokernels by $CL(M)$. This quantity is well defined only up to a discriminant negligible factor.
 We formulate below a simple lemma that we will use in the subsequent sections:

\begin{lemma}\label{triangleinequality}
Let $0\rightarrow M_1\rightarrow M\rightarrow M_2\rightarrow 0$ be an exact sequence of $\Gal(K/\Q)$ modules. Then 
$$\max(CL(M_1),CL(M_2))\leq CL(M)\leq CL(M_1)\cdot CL(M_2).$$ 
\end{lemma}

\begin{proof}
Consider $\Gal(K/\Q)$-lattices $L_1,L_2$ such that $L_1/L_2\approx M$. by pulling back $M_1$ we can find a lattice $L_3$ containing $L_2$ and contained in $L_1$
such that $L_3/L_2\approx M_1$ and $L_1/L_2\approx M_2$. Let $S,T,U$ be the Tori whose cocharacter modules are $L_1,L_2,L_3$ respectively. We have
maps $f:T\rightarrow U$ and $g:U\rightarrow S$ corresponding to the inclusion of the cocharacter modules. This induces maps on class groups:
$$\xymatrix{
CL_T\ar[r]^f& Cl_U\ar[r]^g &CL_S\\
}.$$
By
corollary \ref{transfer} we have $ker(f\circ g) \approx CL(M), ker(f)\approx CL(M_1),$ and $ker(g)\approx CL(M_2)$ up to discriminant negligible factors. 
Moreover, since the Tori are all isogenous their class groups have the same size up to discriminant negligible factors by theorem \ref{BST}.
The result follows.
\end{proof}

As an important special case, suppose $M$ is a finite Galois module with a trivial Galois action. Then $M$ can be expressed as a quotient of $\Z^n$ for some 
integer $n$, where we consider $\Z^n$ as 
a Galois module with a trivial Galois action. $\Z^n$ is the cocharacter module of $(\G_m)^n$, and since the class number of $\Q$ is 1, we conclude the following

\begin{cor}\label{trivial}
Let $M$ be a finite module with a trivial action of the Galois group. Then $CL(M)$ is discriminant negligible. 
\end{cor}

We now give a further interpertation of the class groups of Tori, and of the groups $CL(M)$. Though not necessary for the main results of this paper, we feel
that the next theorem and its corollary shed light on the nature of these objects.

Suppose $T$ is a torus of dimension $d$ in $C(L/\Q), [L:\Q]=n$. Denote the Galois
group of $L$ over $\Q$ by $G$. We then have natural isomorphisms :
\begin{align*}
T(\Q)\cong Hom_G(X(T),T(L))\\
 T(\A_f)\cong Hom_G(X(T),T(\A_{L,f}))\\
 T(\widehat{\Z})\cong Hom_G(X(T),T(\widehat{\Oo}_L)
\end{align*}

We thus have a natural map from $CL_T$ to $Hom_G(X(T),CL(L))$.

\begin{thm}\label{classmap}
Let $T$ be a torus of dimension $d$ in $C(L/\Q)$. Denote the Galois
group of $L$ over $\Q$ by $G$. The natural map from $CL_T$ to $Hom_G(X(T),CL(L))$ has discriminant negligible kernel and cokernel.
\end{thm}

\begin{proof}

Let $U_L$ denote the unit group of $L$. We have the following exact sequence

$$1\rightarrow Hom(X(T),\widehat{\Oo}_L/U_L)\rightarrow Hom(X(T),\A_{L,f}/L^{\times})\rightarrow Hom(X(T),CL(L)) \rightarrow 1$$ 

or equivalently

$$1\rightarrow Hom(X(T),\widehat{\Oo}_L/U_L)\rightarrow T(\A_{L,f})/T(L)\rightarrow Hom(X(T),CL(L)) \rightarrow 1$$

Taking $G$-invariants gives the associated exact sequence 

$$1\rightarrow Hom_G(X(T),\widehat{\Oo}_L/U_L)\rightarrow (T(\A_{L,f})/T(L))^G\rightarrow Hom_G(X(T),CL(L)) \rightarrow H^1(\widehat{\Oo}_L/U_L))\rightarrow 1$$

Now, there is an embedding of $T(\A_f)/T(\Q)$ into $T(\A_{L,f})/T(L))^G$ and  Nakayama-tate duality implies that the cokernel is discriminant negligible
, as in shown in the proof of theorem \ref{torusexact}. We now analyze the second term $\widehat{\Oo}_L/U_L$. Consider the exact sequence

$$1\rightarrow Hom( X(T), U_L)\rightarrow T(\A_{L,f})^c \rightarrow Hom(X(T),\widehat{\Oo}_L/U_L) \rightarrow 1.$$

Taking $G$-invariants gives us the exact sequences

$$ 1\rightarrow T(\Q)\cap T(\hat{\Z})\rightarrow T(\hat{\Z}) \rightarrow Hom_G(X(T),\widehat{\Oo}_L/U_L) \rightarrow H^1(Hom(X(T),U_L))$$

and 

$$H^1(Hom(X(T),U_L)) \rightarrow H^1(T(\A_{L,f})^c) \rightarrow H^1(Hom(X(T),\widehat{\Oo}_L/U_L)) \rightarrow H^2(Hom(X(T),U(L))).$$

As $U_L$ is a finitely generated abelian group of rank bounded by the degree of $L$, all its cohomology groups are of bounded cardinality. The theorem thus 
follows from the following lemma:

\begin{lemma}
The cohomology group $H^1(T(\widehat{\Oo}_L))$ is finite and discriminant negligible.
\end{lemma}

\begin{proof}

Note that we have the decomposition $$H^1(T(\widehat{\Oo}_L))\cong \oplus_v H^1(T(\Oo_{L,v})).$$

Let $v$ be a place lying over a finite prime $p$ such that it is unramified. Then by local class field theory, $\Oo_{L,v}$ is a cohomologically trivial
module, and thus by Nakayama (\cite{Na}, Th. 2) $H^1(T(\Oo_{L,v}))$ is trivial. 

If $v$ is ramified, consider the exact sequence $$1\rightarrow T(\Oo_{L,v})\rightarrow T(L_v)\rightarrow X(T)^*\rightarrow 1.$$ 
By Nakayama-Tate duality, the cohomology groups of $T(L_v)$ are bounded in terms of the degree of $L$. By looking at the associated long exact sequence, 
It thus follows that the same holds for $T(\Oo_{L_v})$.
Let $S$ be the set of primes ramified in $L$. The lemma now follows from the fact that $2^{|S|}$ is discriminant negligible.

\end{proof}
\end{proof}

We have the following corollary to the groups $CL(M)$:

\begin{cor}
Let $M$ be a finite module of the galois group $\GQ$, which factors through the galois group $G$ of a number field $L$. Let $M^*$ denote the dual module
$Hom(M,\Q/\Z)$. Then up to discriminant negligible factors,
we have the equality
$$CL(M)=|Hom_G(M^*,CL(L))|.$$
\end{cor}
\begin{proof}

Let $\phi:T\rightarrow S$ be a map of isogenous tori that split over $L$ such that the cokernel of the induced map $X(T^*)\rightarrow X(S^*)$ is $M$. This implies
that the cokernel of the map $X(S)\rightarrow X(T)$ is $M^*$. 

By theorem \ref{classmap}, up to discriminant negligible factors $CL(M)$ is the size of the kernel of the map 

$Hom_G(X(T),CL(L))\rightarrow Hom_G(X(S),CL(L)),$ which is exactly $Hom_G(M^*,CL(L)).$

\end{proof}

\section{General Transfer Principles for Class Groups}

We can apply corollary \ref{transfer} to derive transfer principles for torsion in class groups of number fields in the following way: 
Take $K$ to be an algebraic number field and consider the multiplication by $n$ map on $T=Res_{K/\Q}\G_m$. 
The cokernel of the induced map on class groups is precisely $CL(K)[n]$, so if we let $M$ be the finite module $X(T)^*/nX(T)^*$ then $CL(M)=CL(K)[n]$.

As an example, consider a cubic field $K$ thats not normal. Let $L$ be its Galois closure, and $k$ the quadratic subfield of $L$. This is called
the \emph{quadratic resolvent} of $K$. Then a result of Gerth \cite{G} says that the 3 torsion in $CL(K)$ and $CL(k)$ are the same up to 
a factor of $O_{\epsilon}(D_K^{\epsilon})$. We can recover this result as follows:

Identify the Galois group $G=\Gal(L/\Q)$ with $S_3$. The 3-torsion in $CL(K)$ corresponds to the $G$-module $M$ generated by $a_1,a_2,a_3$ over $\F_3$ with the 
natural permutation 
action by $S_3$, while the 3-torsion in $CL(k)$ corresponds to a $G$-module $N$, which is isomorphic to the submodule of $M$ generated by 
$a_1-a_2$ and $a_2-a_3$. Moreover, $M/N$ is a trivial $G$-module. 
By corollary \ref{trivial} we have that $CL(M/N)$ is discriminant negligible and so the result follows by lemma \ref{triangleinequality}.
 
Similarly, suppose $K$ is a quartic field whose normal closure $\tilde{K}$ has $S_4$ Galois group over $\Q$. There are 3 subgroups of $S_4$ of order 8, all
conjugate and isomorphic to dihedral group with 8 elements $D_4$. Pick one of these, and let $F$ be its fixed field. $F$ is called
the \emph{cubic resolvent} field of $K$. We have the following result relating 2-torsion in quartic fields to 2-torsion in the corresponding cubic resolvents.
 This result is new as far as the author knows:

\begin{lemma}\label{quartictransfer}
Let $K$ be a quartic field, and $F$ be its cubic resolvent. Then the 2-torsion in $CL(K)$ and in $CL(F)$ are the same up to a factor of size
$O_{\epsilon}(D_K^{\epsilon})$.
\end{lemma}

\begin{proof}

Let $L$ be the Galois closure of $K$ and assume for simplicity that $G=\Gal(L/\Q)\approx S_4$. Let $M$ and $N$ be the $G$-module corresponding to 
$CL(K)[2]$ and $CL(F)[2]$ respectively. As before, $M$ corresponds to the $\F_2$ vector space generated by $a_1,a_2,a_3,a_4$ with the natural permutation
action of $S_4$. We have a 1-dimensional trivial submodule $M_0$ of $M$ generated by $a_1+a_2+a_3+a_4$. Set $M_1$ to be $M/M_0$. Now, let $M_2$ be the submodule 
of $M_1$ generated by $a_1+a_2,a_1+a_3$, so that $M_1/M_2$ is a trivial $G$-module. A basis for $N$ as a module over $\F_2$ is given by $b_1,b_2,b_3$ where the 
$S_4$ action is given by identifying the $b_i$ with the right cosets $S_4/D_4$. We have a trivial 1-dim submodule $N_0$ given by $b_1+b_2+b_3$, and $N/N_0 \approx M_2$ as $S_4$-modules. Now, by applying corollaries \ref{triangleinequality} and \ref{trivial} we get
$$Cl(M)=Cl(M_1)=Cl(M_2)=Cl(N/N_0)=Cl(N)$$ as desired.

\end{proof}

We end this section with some general remarks on torsion in class groups. One expects that all torsion is discriminant negligible. For positive integers $n,d$, Zhang\cite{Z} conjectures the bound 
\begin{equation}\label{epsilonconj}
Cl(K)[n]\ll_{\epsilon} D_K^{\epsilon}
\end{equation}
as $K$ varies over number fields of degree $d$.  We can phrase Zhang's conjecture in our language in the following way: Fix a group $G$ and a finite $G$-module $M$. Now, suppose $K$ is a Galois
number field with $\Gal(K/\Q)\approx G$. We can then consider $M$ as a finite $\Gal(K/\Q)$ module. It then makes sense to talk about $Cl(M)$, leaving $K$ out
by abuse of notation. 

\begin{conj}\label{conjecturetorsion}
For all finite modules $M$, the function $Cl(M)$ is discriminant negligible.
\end{conj}

It is not hard to see that \ref{conjecturetorsion} is equivalent to Zhang's conjecture. Zhang's conjecture is the special case where the module $M$ is the
regular representation of $G$ over $\Z/n\Z$ for some positive integer $n$. However, since all modules are quotients of direct sums of such regular representations,
the equivalence follows from corollary \ref{triangleinequality}.

\begin{itemize}

\item In light of conjecture \ref{conjecturetorsion} the result of Gerth and lemma \ref{quartictransfer} might not seem that interesting, since all torsion
is supposed to be Discriminant negligible! In fact, Gerth\cite{G} proves a much more precise relationship between the 3-torsion group of a cubic field and its 
quadratic resolvent, in terms of ramification.  By following the proof of theorem \ref{torusexact} more closely one can prove that $Cl(M)$ is well
defined up to, roughly speaking, the contribution from the ramified primes in $K$. Thus, the statement of lemma \ref{quartictransfer} could likely be refined to 
say that the 2-ranks of a quartic field and its cubic resolvent are the same up to ramified primes. As our main interest is in Galois orbits of CM points, 
we do not pursue this.

\item In fact, Brumer and Silverman\cite{BS} asked whether the right hand side of \eqref{epsilonconj} can be replaced 
with $D_K^{\frac{C}{\ln\ln D_K}}$. The value $D_K^{\frac{C}{\ln\ln D_K}}$ can be best understood in terms of the ramified primes in $K$. These are precisely the primes dividing $D_K$.
Now, for $S_K$ denoting the number of distinct primes dividing $D_K$, one can verify $$n^{S_K}\ll D_K^{\frac{C}{\ln\ln D_K}},$$ with equality
for $D_K$ equaling the product of the first $X$ primes ($X\approx \frac{\ln D_K}{\ln\ln D_K}$).  Thus, for 2-torsion of certain quadratic fields the bound is 
tight. 
\item The above methods cannot recover some transfer principles, such as the Scholtz reflection principle 
which relates 3-torsion in the quadratic fields $\Q(\sqrt{D})$ and $\Q(\sqrt{-3D})$, or the work of Ellenberg and Venkatesh\cite{EV}. This is because these 
principles use heavily class field theory, which has not been an input in our work thus far. Nonetheless, one can rephrase and generalize these reflection principles in this language. 

\end{itemize}

\section{Applications to Lower bounds for Galois orbits}

In this section we apply results of section 5 to the study Galois orbits of special points in higher rank. Let $A_{g,1}(\C)$ denote the moduli space for 
principally polarized complex abelian varieties of dimension $g$. We also let $\Hh_g$ denote the Siegel-Upper half space, so that  
$$A_{g,1}(\C)\approx GSp_{2g}(\Q)\backslash \Hh_g\times GSp_{2g}(\A_f)/GSp_{2g}(\hat{\Z}).$$ From now on, given a field $F$ we refer to the torus $Res_{F/\Q}\G_m$
 simply as $F^{\times}.$
 
Let $x$ be a CM point in $A_{g,1}$ corresponding to a principally polarized Abelian variety $A_x$.
 We set $R_x$ to be the center of the ring of endomorphisms $End(A_x)$ of $A_x$. All of the CM points in $A_{g,1}(\C)$ are known to be defined over $\bar{\Q}$, 
and so admit a natural action of the Galois group $\GQ$. We recall conjecture \ref{lowerboundscm} for the convenience of the reader:

\begin{conj}\label{general}
There exists a positive constant $\delta_g$ such that the inequality $|\GQ\cdot x|\gg_g \Disc(R_x)^{\delta_g}$ holds for all $\epsilon>0$ as 
$x$ varies over all CM points in $A_{g,1}$. 
\end{conj}

We begin by reviewing the description of the Galois action on $A$ in terms of the CM theory. 

\subsection{Background on CM theory}

Let $A$ be a principally polarized complex CM abelian variety, which we shall take to be simple for the moment. Then there is a CM field $K$ such that
$$End(A)\otimes_{\Z}\Q \cong K.$$  We let $F$ denote the maximal totally real subfield of $K$.
Moreover, there is a set $S=\{\phi_1,\phi_2,\dots,\phi_g\}$ of complex embeddings of $K$ such that the representation
of $K$ on the complex tangent space of $A$ is given as the direct sum of the 1-dimensional representations induced by $S$. Moreover, $S\cup \bar{S}$ give
all embeddings of $K$ into $\C$. Let $L$ denote the normal closure of 
$K$ and $S_0$ be all the embeddings of $L$ into $\C$ which induce an element of $S$ on $K$. The pair $(K,S)$ is referred to as the \emph{CM type} of A.
Since $A$ is simple, $(K,S)$ must be a \emph{primitive} CM-type, which means that the right stabilizer of $S_0$ in $\Gal(L/\Q)$ is $\Gal(K/\Q)$. 

Now, let $MT_x^0$ denote the torus which is the pre-image of the torus $\Q^{\times}$ under the norm map on $K$, so that 
$$MT_x^0:=ker(K^{\times}\rightarrow F^{\times}/\Q^{\times}, a\rightarrow a\bar{a}).$$
Then corresponding to $A$ there is an embedding of algebraic groups: $$\phi_x:MT_x^0\hookrightarrow GSp_{2g},$$ and moreover we have the equality
$$\phi_x(MT_x^0(\Q))\cap GSp_{2g}(\hat{\Z})\approx End(A)^0,$$ where $$End(A)^0:= \{a\in End(A)\mid a\bar{a}\in \Z\}.$$ 

Now, there is a reciprocity morphism $r_x:L^{\times}\rightarrow MT_x^0$ given by $$r_x(a)=\prod_{\tau\in S_0} \tau^{-1}(a).$$ Since the morphism
$r_x$ depends only on the CM type $(K,S)$ we denote it also by $r_{K,S}$. In the future, if there is no confusion we drop the $S$ and simply write $r_K$.
We can consider the map $$s_x:(\A^L_f)^{\times} \rightarrow GSp_{2g}(\Q)\backslash \Hh_g\times GSp_{2g}(\A_f)/GSp_{2g}(\hat{\Z})$$ 
given by  $$s_x(a)\rightarrow (x,\phi_x\circ r_x(a)).$$ 

The image of $s_x$ is equal to the Galois orbit of the CM point representing $A$.
 Moreover, there is a unique arithmetic subtorus $MT_x\hookrightarrow MT_x^0$ such that $r_x$ is surjective onto $MT_x$ as a 
morphism of tori, and we call this the \emph{Mumford-Tate group} of $A$. 

We now review the general theory where $A$ is not necessarily simple. In this case $A$ is isogenous to a product $A_i^{n_i}$ where each $A_i$ is a simple
Abelian variety with $CM$ such the $A_i$ are mutually non-isogenous. Let $(K_i,S_i)$ be the $CM$ type of $A_i$ (Note that it does NOT follow that the $K_i$ are
mutually non-isomorphic!). Let $L$ denote the compositum of the normal closures $L_i$ of the $K_i$, and $S_i^0$ be the set of embeddings of $L$ into $\C$ which 
induce one of the $S_i$ on $K_i$. Let $F_i$ denote the maximal totally real subfield of $K_i$ and set $MT_x^0$ to be the preimage of $\Q^{\times}$ under the norm
map on $\prod_i K_i$ so that $$ MT_x^0:=ker(\prod_i K_i^{\times}\rightarrow (\prod_i F_i^{\times})/\Q^{\times}, a\rightarrow a\bar{a}).$$

As before, there is an embedding $$\phi_x: MT_x^0\hookrightarrow GSp_{2g}$$ and moreover we have the identity 
$$\phi_x(MT^0_x(\Q))\cap GSp_{2g}(\hat{\Z})\approx Z(End(A))^0,$$ where $Z(End(A))$ is the center of $End(A)$ and 
$$Z(End(A))^0:= \{a\in Z(End(A))\mid a\bar{a}\in \Z\}.$$ 

The reciprocity morphism $r_x:L^{\times}\rightarrow MT_x^0$ is now defined to be $$r_x(a)=(\prod_{\tau\in S^i_0} \tau^{-1}(a))_i,$$ where we consider $MT_x^0$ as
a subtorus of $\prod_i K_i^{\times}$. Again, there is a unique arithmetic subtorus $MT_x\hookrightarrow MT_x^0$ such that $r_x$ is surjective onto $MT_x$ as a 
morphism of tori, and we call this the \emph{Mumford-Tate group} of $A$. As before, the image of the composition 
$$s_x: (\A^L_f)^{\times} \rightarrow GSp_{2g}(\Q)\backslash \Hh_g\times GSp_{2g}(\A_f)/GSp_{2g}(\hat{\Z})$$ is the Galois orbit of the CM point representing $A$. 

\subsection{Reducing to fields}

Consider for a moment conjecture \ref{general} in the case where $A$ is a simple $CM$ abelian variety with $CM$ by a field $K$ such that the endomorphism
ring of $A$ is the full ring of integers in $O_K$. Then the map $\phi_x\circ r_x$ factors through $$\widetilde{r_K}:Cl_{L}\rightarrow Cl_{K}.$$ A positive 
answer to conjecture \ref{general} would thus require a positive answer to the following conjecture, which is purely a conjecture about fields:

\begin{conj}\label{specific}
Let $K$ be a CM field of degree 2g, and let $S$ be a primitive CM type for $K$. Let $L$ be the normal closure of $K$, and let $r_K$ denote the reciprocity 
morphism from $L^{\times}$ to $K^{\times}$ corresponding to $S$. Then there exists $\delta(g)>0$ such that 
$$|im(\widetilde{r_K}:CL(L)|\rightarrow CL(K))\gg_{g.\epsilon} \Disc(K)^{\delta(g)-\epsilon}.$$
\end{conj}

Our aim in this section is to show that conjecture \ref{specific} implies conjecture \ref{general}. We then devote the 
remaining sections to the study of conjecture \ref{specific} using the results of sections 4 and 5.

\begin{thm}\label{reduction}
If conjecture \ref{specific} holds for all fields of degree $2h$, with $h\leq g$, then conjecture \ref{general} holds for $A_{g,1}$.
\end{thm}

Let $x$ be a $CM$ point representing a principally polarized abelian variety $A$. We use the notation of section 7.1, so that $A$ is isogenous to 
$\prod_i A_i^{n_i}$, $(K_i,S_i)$ is the CM type of $A_i$, and we set $R_x$ to be the center of $End(A)$. Note that $R$ is in a natural way a subring of 
$\oplus_i O_{K_i}$. For ease of notation, we define $O_x=\oplus_i O_{K_i}$.

Now, let $$K_x=MT_x(\A_f)\cap GSp_{2g}(\hat{\Z})$$ so that $K_x$ is a finite-index subgroup 
of the maximal compact subgroup $MT_x(\hat{\Z})$ of $MT_x(\A_f)$.

By Yafaev (\cite{Y}, proposition 2.8), there is a finite index compact subgroup $K_g$ of $GSp_{2g}(\hat{\Z})$ such that if $K'_x= K_x\cap K_g$, then the map
$$\xymatrix{
MT_x(\Q)\backslash MT_x(\A_f)/K'_x\ar[r]&  GSp_{2g}(\Q)\backslash \Hh_g\times GSp_{2g}(\A_f)/K_g}$$
is injective. It follows that 
\begin{equation}\label{Yinput}
|\GQ\cdot x|\gg_g |im[(\A^L_f)^\times\rightarrow MT_x(\Q)\backslash MT_x(\A_f)/K_x]|
\end{equation}

Now, $R_x$ is naturally a subring of $O_x,$ and $$\Disc(R_x)=[MT_x(\hat{\Z}):K_x]^2\cdot\prod_{i=1}^m D_{K_i}.$$ 

Let $i_x$ denote the set of primes $p$ where $p$ divides $[O_x:R_x]$. 
By Ullmo-Yafaev (\cite{UY}, lemma 2.13), there is a positive constant $B$ depending only on $g$ such that
\begin{equation}\label{UYinput}
|im[(\A_L^f)^\times\rightarrow MT_x(\Q)\backslash MT_x(\A_f)/K_x]|\gg_g B^{|i_x|}\cdot[MT_x(\hat{\Z}):K_x]\cdot |im:Cl(L)\rightarrow Cl_{MT_x}|	
\end{equation}

Let us first deal with the second term. There is a surjection $$\pi_i:MT_x\rightarrow MT_{A_i},$$ and we have the equality
$$ \pi_i\circ r_x = r_{A_i}\circ N_{L/L_i}$$ where $N_{L/L_i}$ is the norm map from $L$ to $L_i$.  By class field theory the norm map on class groups has cokernel
bonded by the dimension of the fields. Since we are also assuming conjecture \ref{specific}, it follows that 
$$\forall i,|im: CL(L)\rightarrow Cl_{MT_x}| \gg_g D_{K_i}^{\delta(g_i)}$$ where $g_i$ is the dimension of $A_i$. 

In particular, this proves the theorem unless the index $[O_x:R_x]$ grows faster than any power of the discriminants $D_{K_i}$. Thus, 
the theorem will follow once we prove the following

\begin{lemma}\label{localestimate}
There are constants $c_g>0,d_g$ such that $$[O_x:R_x]\leq [MT_x(\hat{\Z}):K_x]^{c_g}\cdot \left(\prod_{i=1}^m D_{K_i}\right)^{d_g}.$$
\end{lemma}

\begin{proof}

We must first show that is impossible for $R_x$ to decrease without $K_x$ also decreasing. To that end, we have the following elementary lemma:
\begin{lemma}\label{generate}
$MT_x(\Q)$ generates $\oplus_i K_i$ as a ring.
\end{lemma}

\begin{proof}

First, it follows from $S_i$ being a primitive CM type that the only elements of $\Gal(L_i/\Q)$ which fix $r_{A_i}(a)$ for all $a\in L_i$ are in 
$\Gal(L_i/K_i)$. Thus
$MT_x(\Q)$ generates a $\Q$-algebra $E$ in $\oplus_i K_i$ which surjects onto $K_i$ for each $i$. Likewise, since the $S_i$ are pairwise inequivalent, the only
elements of the Galois group of the Galois algebra $\oplus_i L_i$ which fix $E$ are those in the group $\oplus_i \Gal(L_i/K_i)$. The lemma follows. 
\end{proof}

For $\alpha\in L(\Z_p)$, define $$S(\alpha):= \sum_{\tau\in S} \tau^{-1}(\alpha).$$ Note that we have the identity $$r_x(e^{p\alpha}) = e^{pS(\alpha)}.$$ We shall
need the following more refined version of lemma \ref{generate}:

\begin{lemma}\label{generatewell}
There is a constant $C$ depending only on $g$ such that we can pick $\alpha\in L(\Z_p)$ with $\Z_p[S(\alpha)]$ having index at most $p^C$ in $O_x$.
\end{lemma}

\begin{proof}

Let $\psi_i$ denote all the ring homomorphisms from $(O_x)_p$ to $\overline{\Q_p}$ over $\Q_p$. Then the index of $\Z_p[S(\alpha)]$ in $O_x$ is given by
$v_p(i(\alpha)$ where 
$$i(\alpha)=\prod_{i\neq j} (\psi_i(S(\alpha))-\psi_j(S(\alpha)).$$ 

Now, we consider $$\psi_{ij}=\psi_i\circ S -\psi_j\circ S$$ as a linear map from 
$L$ to $L$. The lemma will follow once we show that there is uniform constant $C$ and an element $\alpha\in L(\Z_p)$ such that $v_p(\psi_{ij})\leq p^C$.
To prove this, recall that if we take a $\Z_p$ basis $\beta_i$ for $L(\Z_p)$ and let $(\sigma_1,\dots,\sigma_m)$ be the elements of the Galois group
of $L_p$ then the $p$-valuation of the determinant of the matrix $(\sigma_j(\beta_i))$ is the discriminant of $L_p$ and therefore bounded by $p^C$ for some 
absolute constant $C$. 
If we replace $\sigma_1$ in the matrix by $\psi_{ij}$ then the determinant scales at most by the $p$-power of the coefficient of $\sigma_1$ in
$\psi_{ij}$, which is uniformly bounded. This finishes the proof.

\end{proof}

Since $$[MT_x(\hat{\Z}):K_x]=\prod_{p} [MT_x(\Z_p):(K_x)_p]$$ and 
$$[O_x:R_x]=\prod_{p}[\prod O_{K_i}\otimes_{\Z}\Z_p:R_x\otimes_{\Z}\Z_p]$$ the question is essentially reduced to a local one. 

By the proof of lemma 3.12 in \cite{UY}, if $p$ is unramified in $L$ and $p|[O_x:R_x]$ then
$$[MT_x(\Z_p):(K_x)_p]\gg p.$$ To finish the proof, we therefore need only to consider the case where 
$$[O_x:R_x]$$ is divisible by a very large power of some prime $p$.
The lemma \ref{localestimate} and theorem \ref{reduction} thus follows from the following lemma:

\begin{lemma}
There is an integer $m$ depending only on $g$ such that if $p^n || [MT_x(\Z_p):(K_x)_p]$,  then $p^{m(n+1)}$ does not divide $[O_x:R_x]$.
\end{lemma}

\begin{proof}

Pick $\alpha$ as in lemma \ref{generatewell} so that the index if $\Z_p[S(\alpha)]$ in $O_x$ is bounded by $p^C$. Now set $h$ be the dimension of $O_x$ and 
define $\beta_i:= e^{ip^N S(\alpha)}$ for $0\leq i\leq h-1$. Note that $\beta_i\in K_x$.

$$\beta_i = 1+ip^NS(\alpha) + \frac{i^2p^{2N}S(\alpha)^2}{2!} +\cdots + \frac{i^{h-1}p^{N(h-1)}S(\alpha)^{h-1}}{(h-1)!} + p^{Nh - h}\cdot (O_x)_p.$$
The matrix $$M_h = (\frac{i^j}{j!})_{0\leq i,h\leq h-1}$$ is invertible with some determinant $m_h$. Thus, if $N>n+C+m_h$ we see that the span of the 
$\beta_i$ contains $p^{Nh-h-1}\cdot ((O_x)_p/p\cdot (O_x)_p)$ 
and thus $p^{Nh-g-1}(O_x)_p$ by Nakayamas lemma. But this means that $$[O_x: R_x]\leq p^{Nh^2-h^2-h},$$ which finishes the 
proof.

\end{proof}

\end{proof}

\subsection{General Lower bounds under the Generalized Riemann Hypothesis}
$\newline\newline$
Assuming the Generalized Riemann Hypothesis for CM fields, Yafaev\cite{Y} proved a lower bound for Galois orbits of special points of the form $c_Nlog(D_K)^N$
for any $N>0$. The idea of his proof was to use small split primes to generate a large number of distinct elements of $CL(K)$. We present here
a slight refinement of his method which allows us to get a polynomial lower bound in $D_K$, thus showing that GRH implies a positive answer to question 1.1
in general. Yafaev worked over a general Shimura variety. For simplicity of exposition, we shall only present our method for the case of Siegel space $A_{g,1}$, 
as this allows us to work classically, although the method applies just as well for any Shimura variety. The idea is that while Yafaev used very small 
split primes to {\it generate} a large amount of class group elements, we use much larger split primes (a small power of the discriminant) to 
{\it represent} lots of distinct elements in the class group.

\begin{thm}\label{GRH}
Assume GRH for CM fields. Then 
$$|im(\widetilde{r_K})|\gg_{g,\epsilon}D_K^{\frac{1}{2g^2}-\epsilon}.$$
\end{thm}

\begin{proof}

Let $\delta$ denote the element of $\Gal(L/\Q)$ representing complex conjugation. Consider unramified primes $p,q$ which splits completely in $K$, and therefore also in $L$. Choose prime ideals $P,Q$ in $O_L$ 
above $p,q$ respectively, and suppose that $r_K(P)=r_K(Q)$ as elements of the ideal class group $CL(K)$. Then there is an element $x\in K$ such that
$$(x) = \frac{r_K(P)}{r_K(Q)},$$ and hence $$(x/x^{\delta})=\frac{r_K(P)r_K(Q^{\delta})}{r_K(Q)r_K(P^{\delta})}$$ Since $K$ is a primitive CM type and $P$ and $Q$
 are totally split, $\frac{x}{x^{\delta}}$ must generate $K$ over $\Q$. Now, consider the element $$y=\frac{pqx}{x^{\delta}}.$$ Then $y$ is an algebraic integer
which generates $K$ over $\Q$. Moreover, all conjugates of $y$ are of absolute value $pq$. Since $\Disc(\Z[y])\geq D_K$ we arrive at the inequality
$$pq\gg_g D_K^{\frac{1}{g^2}}.$$ 

It follows that $|im(r_K)|>\{\#{p\mid\textrm{ p splits completely in } K \textrm{ and } p\ll_g D_K^{\frac{1}{2g^2}}}\}$. Since GRH implies the existence of at 
least $D_K^{\frac{1}{2g^2}-\epsilon}$ such primes, we are done.

\end{proof}

\subsection{Setting up the combinatorics}

$\newline\newline$
For each $g$, define $W_g$ to be the Weyl group of order $g$,
which we will view as a group of signed permutations. That is, an element of $W_g$ is a permutation $\sigma$ in $S_g$ together with a choice, for each $i$,
of sign $\sigma_i$. If $\sigma(i)=j$, then we see say that $\sigma$ takes $i$ to $j$ with sign $\sigma_i\in\{+1,-1\}$. We define composition by 
multiplying signs. That is, $(\sigma\circ\tau)(i) = \sigma(\tau(i))$ and $(\sigma\circ\tau)_i = \sigma_{\tau(i)}\cdot \tau_i$. It is convenient to use the 
following notation for an element of $W_g$: the element $\sigma=(1^+3^-)(2^-)$ is used to denote the permutation $(13)(2)$ which takes $1$ to $3$ with positive 
sign, $3$ to $1$ with negative sign and $2$ to itself with negative sign. So for instance, $\sigma^2 = (1^-)(3^-)(2^+)$. Note that there is a map 
$W_g\rightarrow S_g$ taking a signed permutation to the underlying permutation, and this map is a surjective homomorphism. The kernel of this map will be denote
by $W_g^{0}$. Thus $W_g^{0}$ is a normal subgroup of $W_g$ isomorphic to $(\Z/2\Z)^g$. 

As before, Let $(K,\Sigma_K)$ denote a $CM$ field of degree $g$ together with a CM type 
$(\Sigma_K	=(\phi_1,\dots,\phi_g)$. Let $L$ be the Galois closure of $K$, and $G$ be the Galois group $\Gal(L/\Q)$. Let $S\in\Gal(L/\Q)$ denote all
elements inducing an element of $\Sigma_K$ on $K$. We also let $\delta\in G$ be complex conjugation. Then $G$ of $L$ can canonically be viewed as a subgroup
of $W_g$ in the following way: For each $g\in G$ and $1\leq i \leq g$, we have that $g(\phi_i)$ is either $\phi_j$ for some $j$, or $\overline{\phi_j}$ for some
$j$. Thus, we can define $\sigma_g$ to be the signed permutation mapping $i$ to $j$ with sign 1 if $g(\phi_i)=\phi_j$ and sign -1 if $g(\phi_i)=\overline{\phi_j}$.
It is easy to see that this is an injective homomorphisms $\Gal(L/\Q)\hookrightarrow W_g$. It is worthwhile noting under this map complex conjugation always maps 
to the element $(1^-)(2^-)\dots(g^-)$. We assume wlog that $\phi_1$ is the identity embedding on $K$. 
It is follows under this identification that $S$ consists of the set of all elements of $G$ that send the element $1$ to some $j$ with 
sign 1, and thus $S^{-1}$ is the set of all elements which send some $i$ to $1$ with sign 1. Moreover, $H$ is the subgroup of $G$ which sends
$1$ to itself with sign 1. Since we assumed that the abelian variety $A$ is simple, the right stabilizer of $S$ in $G$ must be $H$. 

\subsection{Understanding the map $r_K$ on Characters and cocharacters}
$\newline\newline$
Let us being by understanding the relevant $G$ modules. A basis for $X(K^{\times})$ is the set of all embeddings $\psi_i:K\rightarrow \C$, where the Galois action
is given by composition: $g(\psi_j) = g\circ\psi_j$. Likewise, A basis for $X(L^{\times})$ is the set of all embeddings $\xi_j:L\rightarrow \C$, where the Galois 
action is again given by composition. By fixing an identity embedding $L\hookrightarrow\C$ we can consider the elements $g$ of $G$ as a basis for 
$X(L^{\times})$. The map $\widehat{r_K}$ sends $\psi_i$ to
$$ \widehat{r_K}(\psi_i) = \psi_i\circ\left(\sum_{g\in S^{-1}} g\right)$$ where we consider $g$ now as an element of $X(L^{\times})$. It will be more natural
to consider the map on cocharacters for us. We take the dual bases $g^*$ for $X(L^{\times})^*$ and $\psi_i^*$ for $X(K^{\times})^*$. The dual map
$\widehat{r_K}^*$ takes an element $g^*$ to $$\widehat{r_K}^*(g^*)=\sum_{1\leq i\leq g} \delta^{\frac{(1-\sigma_i)}{2}}(\phi_i^*)$$
where $\sigma$ is the element of $W_g$ corresponding to $g^{-1}$. That is, $\phi_i^*$ appears in the sum if and only if $g$ takes some element $j$ to $i$
with sign $1$.

Denote by $L_0$ the kernel of the norm map $Nm_L:L^{\times}\rightarrow \Q^{\times}$ and define $K_0$ to be the kernel of the Norm map from $K$
to its totally real subfield $F$ given by  $$Nm_{K/F}(k)=k\bar{k}.$$

$X(L_0)$ is the quotient of $X(L^{\times})$ by
the primitive vector $\sum_{g\in G} g$ and $X(L_0)^*$ is the submodule of $X(L^{\times})^*$ consisting of all elements $\sum_{g\in G}a_g\cdot g^*$ such that 
$\sum_g a_g=0$. Similarly, $X(K_0)$ is the quotient of $X(K^{\times})$ by the relations $\phi_i=-\bar{\phi_i}$ and 
$X(K_0)^*$ is spanned by the elements $\pi_i^*:=\phi_i^*-\bar{\phi_i^*}, 1\leq i\leq g$. Note that $r_K$ takes $L_0$ to $K_0$.

Define $M$ to be the image of $X(K_0)$ in $X(L_0)$ by $\widehat{r_K}$ and let $T_M$ be the torus whose character lattice is $M$.
 
\begin{lemma}\label{bigimage}

We have $D_{T_M}\gg_g D_K^{\frac{1}{2g-1}}.$

\end{lemma}

\begin{proof}

Since $x$ corresponds to a simple abelian variety, the $CM$ type $(K,\Sigma_K)$ must be primitive and so the stabilizer of
the element $\sum_{g\in S^{-1}} (g-\bar{g})$ in $G$ is $H$. Note that this element lies in $M$. Since $L$ is the Galois closure of $K$ we know that
$H$ contains no non-trivial normal subgroup of $G$, so we see that the action of $G$ on $M$ has trivial kernel. The result then follows by considering the Artin 
conductor of the associated representation and applying Theorem \ref{BST} and Lemma \ref{artin}.

\end{proof}

Note that $T_M(\R)$ is compact, so $T_M$ has trivial regulator. In particular, by theorem \ref{torusexact} we can conclude
that the Galois orbit of $x$ is large as long as $M$ is a primitive sublattice of $X(L_0)$. 
This is equivalent to the image of $X(L_0)^*$ in $X(K_0)^*$ being primitive. This is the description we shall work
with.

\subsection{Lower bounds for Galois orbits of Weyl CM-fields}
$\newline\newline$
Following Chai and Oort \cite{CO}, we say that $K$ is a Weyl CM field if the Galois group $\Gal(L/\Q)$ is isomorphic to $W_g$. Weyl CM fields
can be thought of as the generic CM fields, much like a generic field of degree $g$ has Galois group $S_g$. 
The following lemma combined with theorem \ref{reduction} proves theorem \ref{Weyl}.

\begin{lemma}
IF $K$ is a Weyl CM field , then we have $$|im(\widetilde{r_K})|\gg_{g,\epsilon}D_K^{\frac14-\epsilon}.$$
\end{lemma}

\begin{proof}
Since $K_0$ fits into an exact sequence of Tori $$F^{\times}\rightarrow K^{\times}\rightarrow K_0$$ and the Artin conductor is multiplicative, it follows
from corollary \ref{discartin} that $$D_{K_0}\gg_{\epsilon}\frac{D_K^{1-\epsilon}}{D_F}\geq D_K^{1/2},$$ so the lemma will follow once we show that 
$\widehat{r_K}^*$ maps $X(L_0)^*$ surjectively onto $X(K_0)^*$. 

For $1\leq i\leq g$, let $g_i\in\Gal(L/\Q)$ be the element such that
\begin{equation*}
g_i(\phi_j)=
\begin{cases}
\phi_j & i\neq j\\
\bar{\phi_i} & i=j\\
\end{cases}
\end{equation*}

Note that we know $g_i$ exists precisely because $K$ is a Weyl CM field. Now, $1^*-g_i^*$ is in $X(L_0)^*$ and maps onto $\pi_i^*$ in $X(K_0)^*$. Since
the $\pi_i^*$ span $X(K_0)^*$, this completes the proof.

\end{proof}

\subsection{Proofs for g=2,3,4,5, and 6}
$\newline\newline$
In this subsection we prove that Galois orbits are large for $g=2,3,4,5,6$. We record one final combinatorial fact before
proceeding:

\begin{lemma}\label{transport}
For each $i\neq 1$ and $\epsilon\in\{+1,-1\}$ there exist an element $g$ in $S$ which takes $i$ somewhere with sign $\epsilon$. 
\end{lemma}

\begin{proof}

Suppose for the sake of contradiction that there is an $i\neq 1$ and an $\epsilon\in\{+1,-1\}$ 
such that $\forall g\in S$, $g$ takes $i$ to $g(i)$ with sign $\epsilon$. Since $G$ acts transitively on the set $\{1,2,\dots,g\}$ there is a $g\in G$ with
$g(1) = i$. Multiplying $g$ by $\delta$ if necessary, we can also ensure that $g$ takes $1$ to $i$ with sign $\epsilon$. But then $g$ is in the right stabilizer
of $S$ and $g\notin H$, contradicting the fact that $S$ induces a primitive CM type on $K$.

\end{proof}

The following theorem combined with theorem \ref{reduction} proves theorem \ref{main}.

\begin{thm}

Conjecture \ref{specific} holds for $g=2,3,4,5$ with $$\delta(2)=\delta(3)=\frac14,$${} $$\delta(4)=\frac17, $$ $$\delta(5)=\frac19,$$ 
and $$\delta(6)=\frac1{12}.$$

\end{thm}

\begin{remark}
\emph{ While $\delta(2)$ and $\delta(3)$ are optimal, we have made no attempts to optimize $\delta(4), \delta(5)$ or $\delta(6)$. Since the proof reduces '$\delta(4)$ to
a discriminant calculation, optimizing $\delta(4)$ is a straightforward exercise in representation theory, while optimizing $\delta(5)$ would involve
bounding 3-torsion in class groups of quadratic fields. Of course, optimizing the values $\delta_g$ for conjecture \ref{general} would require more work, 
and in particular a better way to handle the case of non-maximal endomorphism ring. }
\end{remark}

\begin{proof}

The theorem follows from lemmas \ref{g2}, \ref{g3}, \ref{g4}, \ref{g5}, \ref{g6} and the previous lemma \ref{bigimage}.

\begin{lemma}\label{g2}

For $g=2$, the image of $\widehat{r_K}^*$ is all of $X(K_0)^*$.
\end{lemma}

\begin{proof}

$X(K_0)^*$ is generated by $\pi_1^*$ and $\pi_2^*$. We write $(a,b)$ to denote the element
$a\pi_1^*+b\pi_2^*$. Recall that $X(L_0)^*$ is generated by the elements $g^*-1^*$. 

Now, we have $$\widehat{r_K}^*(\delta^*-1^*) = (-1,-1).$$ Moreover, by lemma \ref{transport}
there is an element in $S$ which sends $1$ somewhere with sign 1 while sending $2$ somewhere
with sign $-1$. This must either be $(1^+2^-)$ or $(1^+)(2^-)$. So either $(0,-1)$ or $(-1,0)$ is in the image of $\widehat{r_K}^*$.
This together with  $(-1,-1)$ already generates so we're done.
\end{proof}

\begin{lemma}\label{g3}

For $g=3$, the image of $\widehat{r_K}^*$ is all of $X(K_0)^*$.

\end{lemma}

\begin{proof}

Similarly to the case $g=2$, we write $(a,b,c)$ to denote $a\pi_1^*+b\pi_2^*+c\pi_3^*$. We have $$\widehat{r_K}^*(\delta^*-1^*)=(-1,-1,-1).$$

Since $G$ acts transitively on the embeddings $\phi_1,\phi_2,\phi_3$ it must contain a 3-cycle. 
Suppose $G$ has a 3-cycle $\sigma$ where not all the signs are the same. By multiplying $\sigma$ with complex
conjugation $\delta$ we can ensure that two of the signs are positive. Assume $\sigma=(1^+2^+3^-)$. Then 
$$\widehat{r_K}^*(1^*-\sigma^*) = (1,0,0)$$ and $$\widehat{r_K}^*(1^*-(\sigma^2)^*)=\widehat{r_K}^*(1^*-(1^+3^-2^-)^*)=(1,1,0)$$ and these 2 vectors
combined with $(-1,-1,-1)$ generate $X(K_0)^*$.

The same logic holds for the other such cycles. 
Thus, we are reduced to the case where $G$ only contains 3-cycles where all signs are the same. 
Since $G$ contains $\delta$ we see that $G$ must contain
all 3 cycles where all signs are the same, in particular the element $\sigma=(1^+2^+3^+)$. 

Recall that the image of $\widehat{r_K}^*$ is closed under the action of $G$.
By lemma \ref{transport} $G$ contains an element $\tau$ which takes something to $1$ with a positive sign and 
something to $2$ with a negative sign. Hence,
$\widehat{r_K}^*(1^*-\tau^*)$ is either $(0,1,0)$ or $(0,1,1)$. Now, acting by $\sigma$ we see that the image of $\widehat{r_K}^*$ contains either all of
$(0,1,0), (0,0,1), (1,0,0)$ or $(0,1,1), (1,0,1), (1,1,0)$. Either way, together with $(1,1,1)$ we see that this generates $X(K_0)^*$.

\end{proof}

\begin{lemma}\label{g4}

For $g=4$, the size of the image of $\widetilde{r_K}$ is at least $c_{\epsilon}D_K^{\frac17-\epsilon}.$

\end{lemma}

\begin{proof}

Again, we use the notation $(a,b,c,d)$ to denote $a\pi_1^*+b\pi_2^*+c\pi_3^*+d\pi_4^*$. Let $$N_K:=\{(a,b,c,d)\mid a+b+c+d\equiv 0\mod{2}\}.$$
Now, as $X(L_0)^*$ is generated by the elements $1^*-g^*$, we see
that the image of $\widehat{r_K}^*$ is the span of some subset $W$ of the vectors $$\{(a_1,a_2,a_3,a_4)\mid a_i\in\{0,1\}, 1\leq i\leq 4\}$$ which
includes the vector $v=(1,1,1,1)$.  We claim that the image of $\widehat{r_K}^*$ is either primitive, or the sublattice $N_K$. 

To see this, note first that by replacing $w$ by $v-w$ we can assume all vectors in $W$ with 1 or 2 entries being 1, and the remaining
entries being 0. Now, if $W$ contains a vector with a single entry being 1, then since the image of $\widehat{r_K}^*$ is $G$-invariant it must contain all the 
basis vectors and thus be the whole thing. We can thus assume that $W$ only contains vectors with exactly 2 entries being 1. These come in pairs 
$$\{(1,1,0,0),(0,0,1,1)\},\{(1,0,1,0), (0,1,0,1)\}, \{(1,0,0,1),(0,1,1,0)\}.$$
If $S$ contains 1 or 2 such pairs than the span can easily be checked to be primitive. Hence the only case where the image is not a 
primitive sublattice is if $S$ contains all 3 such pairs, and the image is thus $N_K$. 

Now, if the image is primitive the lemma follows by lemma \ref{bigimage}. We thus restrict to the case where the image is $N_K$. Let $T_K$ be the torus whose 
cocharacter module is $N_K$, and consider the induced map $T_M\rightarrow K_0$. It is sufficient to 
show that the cokernel of the induced map on Class groups is Discriminant negligible. But $X(L_0)^*/N_K \approx \Z/2\Z$, and so it must be a trivial Galois
module. The result therefore follows from corollary \ref{transfer}. 

\end{proof}

\begin{lemma}\label{g5}

For $g=5$, the size of the image of $\widetilde{r_K}$ is at least $c_{\epsilon}D_K^{\frac19-\epsilon}.$

\end{lemma}

\begin{proof}

We use the notation $(a,b,c,d,e)$ to denote $a\pi_1^*+b\pi_2^*+c\pi_3^*+d\pi_4^*+e\pi_5^*$. Now, if the span of $\widehat{r_K}^*$ is primitive, we're done by
lemma \ref{bigimage}, so we assume from now on that this is not the case. As $X(L_0)^*$ is generated by the elements $1^*-g^*$, we see
that the image of $\widehat{r_K}^*$ is the span of some subset $W$ of the vectors $$\{(a_1,a_2,a_3,a_4,a_5)\mid a_i\in\{0,1\}, 1\leq i\leq 5\}$$ which
includes the vector $v=(1,1,1,1,1)$. By replacing $w$ by $v-w$ we can assume all vectors in  $W$ except $v$ have 1 or 2  entries being 1, and at least 3 
entries being 0. If any vectors in $W$ have exactly 1 component being 1 then this eliminates a variable and reduces us to the $g=4$ case above, so we can
assume all vectors in $W$ have exactly 2 entries being 1.  Thus the image contains a vector $\pi_z^*+\pi_w^*$ for some $z\neq w$. Now, note that the image
of $\widehat{r_K}^*$ is closed under the action of $G$. Moreover, since $G$ contains an element of order $5$, it must contain a 5-cycle. Taking a power
of this 5 cycle we can conclude that $G$ contains a 5-cycle $\sigma$ which takes  $z$ to $w$ (and has some signs). Acting on the vector $\pi_i^*+\pi_j^*$ 
by powers of $\sigma$ we see that the cokernel $B$ is generated by $\pi_z^*$ with each $\pi_k^*$ equaling  $\pm\pi_z^*$. We also have the relation 
$\pi_1^*+\pi_2^*+\pi_3^*+\pi_4^*+\pi_5^*=0$ in $B$. Since $\pi_w^* = -\pi_z^*$, we can conclude that either $B$ is trivial or $B\approx \Z/3\Z$. Since we assumed
that $B$ isn't trivial, we must have $B\approx \Z/3\Z$. 

Our goal is now to bound $CL(B)$. The automorphism group of $B$ as an abelian group is $\Z/2\Z$, so we get a map from $\Gal(L/\Q)$ to $\Z/2\Z$. The kernel $G_1$ 
must be of index 2.
Now consider the quadratic field $k$ fixed by $G_1$. Notice that $k\subset K$ since $\Gal(L/K)$ fixes $\pi_1^*$ and thus acts trivially on $B$. Now, the module
$X(k^{\times})^*/2X(k^{\times})^*$ is an extension of $B$ by a trivial Galois module, so by corollary \ref{transfer} we have that $CL(B)=CL(k)[2]$ up to 
discriminant negligible factors. In particular, $$CL(B)\leq |CL(k)|\leq D_k^{\frac12}.$$

The rest is easy: since $k\subset K$, we have $D_k^5\leq D_K$. Also, $D_{K_0}\gg_{\epsilon}D_K^{1-\epsilon}/D_F$ where $F$ is the maximal totally real subfield 
of $K$, so $D_F^2\leq D_K.$ Thus 
$$D_{K_0}\gg_{\epsilon} D_K^{-\epsilon}\cdot\frac{D_{K}^{\frac15}}{D_k^{\frac12}}\geq D_K^{\frac3{20}-\epsilon}> D_K^{\frac19}$$
as desired.

\end{proof}

\begin{lemma}\label{g6}

For $g=6$, the size of the image of $\widetilde{r_K}^*$ is at least $c_{\epsilon}D_K^{\frac1{12}-\epsilon}.$

\end{lemma}

\begin{proof}

Let $B$ be the cokernel of $\widehat{r_K}^*$. As $X(L_0)^*$ is generated by the elements $1^*-g^*$, we see
that the image $U$ of $\widehat{r_K}^*$ is the span of some subset $W$ of the vectors $$\{\sum_{i=1}^6 a_i\pi_i^*\mid a_i\in\{0,1\}, 1\leq i\leq 6\}$$ which
includes the vector $v=(1,1,1,1,1,1)$. Given a vector $w=\sum_{i=1}^6 c_i\pi_i^*$, we refer to $c_i$ as the $i$'th entry of $w$. By replacing $w$ by $v-w$ we 
can assume all vectors in  $W$ with 1,2 or 3  entries being 1, and at least 3 
entries being 0. Recall that $U$ is invariant by $G$. Since $G$ acts transitively on the vectors $\pm\pi_i^*$, we deduce that the
order of the $\pi_i^*$ (possibly infinite) in $B$ is independent of $i$. If $B$ is torsion free then the lemma follows from lemma \ref{bigimage}, so we assume
from now on that this isn't the case.

\begin{lemma}\label{transpose}

The image of $\widehat{r_K}^*$ contains an element of the form $\pi_i^*\pm \pi_j^*$, for $j\neq i$.
\end{lemma}

\begin{proof}

If $W$ contains a vector with exactly 1 entry being $1$, then for some $i$ we have $\pi_i^*=0$ in $B$, and hence by the $G$-invariance of $U$ we see
that $B$ must be trivial, contradicting our assumption that $B$ is torsion free. If $W$ contains an element with exactly 2 entries being $1$'s, then we're done. 
Hence the only remaining case is where all the vectors in $W$ have exactly 3 entries being $1$'s. Let $v_1$ be 1 such vector. If $W$ consists only of $v_1$ and/or $v-v_1$, then it is easy to see that $B$ is 
isomorphic to $\Z^4$ and hence torsion free. So $W$ must contain at least 2 vectors $v_1,v_2$ with 3 non-zero entries and $v_1\notin\{v_2,v-v_1\}$. Then
either $v_1$ and $v_2$ or $v_1$ and $v-v_2$ have 2 of the same entries being $1$. Assume wlog $v_1$ and $v_2$ have 2 of the same entries being $1$. 
Then $v_1-v_2$ has 1 entry equal to 1, 1 entry equal to -1, and the rest equal to 0. This proves the lemma.
\end{proof}

By acting by $G$ on the vector produced by lemma \ref{transpose} we deduce that $U$ contains an element of the form $\pi_1^*\pm \pi_j^*$, and by relabeling we can assume
it contains $\pi_1^*\pm \pi_2^*$. Now, $W_6$ acts naturally on $X(L_0)^*
$ and $U$ is invariant by $G$. Since $U$ contains $\pi_1^*\pm \pi_2^*$, we see that $U$ is also invariant
either by $(1^+2^+)$ or $(1^-2^-)$. Let $G^U$ denote the subgroup of $W$ preserving $U$, so that $G^U$ contains $G$ and this signed transposition.
Let $G^0$ denote the image of $G^U$ in $S_6$, so that $G^0$ is a transitive subgroup of $S_6$ containing $(12)$. Draw a graph $V_U$ on $\{1,2,3,4,5,6\}$ where two elements
$i,j$ are connected by an edge if $(ij)\in G^0$. Since $(ij)$ and $(jk)$ generate $(ik)$, we see that this graph is a disjoint union of
complete graphs. Since $G^0$ acts on this graph transitively, the complete graphs must have the same size. 
We thus see that either $V_U$ is the complete graph, in which case $G^0$ is either all of $S_6$, or  $V_U$ is 2 triangles or 3 disjoint edges. We thus split into 3 cases:

\begin{itemize}

\item \textbf{Case i:$G^0\approx S_6.$}$\newline\newline$

The $G^U$ orbit of $\pi_1^*\pm \pi_2^*$ thus contains either $\pi_1^*+\pi_j^*$ or $\pi_1^*-\pi_j^*$ for each $j\neq 1$. Thus $B$ is cyclic and generated by $\pi_1^*$. Since we assumed that 
$B$ is torsion, we see that $U$ has full rank in $X(L_0)^*$. Now, each element of $G$ takes $\pi_1^*$ to $\pm \pi_j^*$ for some $j$. As $\pi_j^*=\pm \pi_1^*$ in $B$,
we see that the $G$ action on $B$ factors through $\Z/2\Z$. Taking the kernel of the action of $G$ on $B$, we see that there is a quadratic field $k\in K$
such that $Cl(B)\leq CL(k)$. But $CL(k)\ll_{\epsilon} D_k^{\frac12+\epsilon}\ll D_K^{\frac1{12}+\epsilon}$. Thus,
$$|im(\widetilde{r_K})|\gg_{\epsilon}\frac{D_K^{\frac12-\epsilon}}{D_F^{\frac12}D_k^{\frac12}}\gg D_K^{\frac16-\epsilon}.$$

\item \textbf{Case ii: $V_U$ consists of 2 triangles with vertices $\{1,2,3\}$ and $\{4,5,6\}$.}$\newline\newline$

In this case $\{\{1,2,3\},\{4,5,6\}$ must be a $G^0$-quotient set.

As in Case i, we have that in $B$, $$\pi_1^*=\pm \pi_2^*=\pm \pi_3^*, \pi_4^*=\pm\pi_5^*=\pm\pi_6^*.$$ Let $G_0$ denote the stabilizer of $\pi_1^*$
in $G$. Then $G_0$ contains $H$ as a subgroup of index at most $3$, so that there is a field $k\in K$ of degree at most $4$ over $\Q$ 
such that $Cl(B)\leq CL(k)$. But $CL(k)\ll_{\epsilon} D_k^{\frac12+\epsilon}\ll D_K^{\frac1{6}+\epsilon}$. Thus,
$$|im(\widetilde{r_K})|\gg_{\epsilon}\frac{D_K^{\frac12-\epsilon}}{D_F^{\frac12}D_k^{\frac12}}\gg D_K^{\frac1{12}-\epsilon}.$$

\item \textbf{Case iii: $V_U$ consists of 3 disjoint edges $\{1,2\},\{3,4\}$ and $\{5,6\}$.}$\newline\newline$

In this case $\{\{1,2\},\{3,4\},\{5,6\}\}$ must be a $G^0$-quotient set. Also, we have $$\pi_1^*=\pm\pi_2^*,\pi_3^*=\pm\pi_4^*, \pi_5^*=\pm\pi_6^*$$ in $B$.

Suppose that $\pi_1^*$ is 2-torsion in $B$. Then the stabilizer $G_0\subset G$ of $\pi_1^*$ in $B$ has $H$ as a subgroup of index at least $4$, and so
there is a subfield $k\subset K$ of degree at most $3$ such that $$Cl(B)\leq CL(k)\ll_{\epsilon} D_k^{\frac14+\epsilon}.$$ Thus
$$|im(\widetilde{r_K})|\gg_{\epsilon}\frac{D_K^{\frac12-\epsilon}}{D_F^{\frac12}D_k^{\frac12}}\gg D_K^{\frac18-\epsilon}.$$
Henceforth we assume that $\pi_i^*$ is not 2-torsion in $B$. 

 Consider a vector $w\in W$. If $w$ has $1$ non-zero entry then $B$ 
is trivial as before and we're done. If $w$ has 2 non-zero entries,
then they must both be $\{1,2\}$, $\{3,4\}$ or $\{5,6\}$ by the main assumption of this case. If $W$ consists only of some subset of 
$$\{\pi_1^*+\pi_2^*, \pi_3^*+\pi_4^*,\pi_5^*+\pi_6^*\}$$ then since $U$ is generated by $W$ and $v$ and $W\cup v$ can be completed to a basis of $X(L_0)^*$,
 it follows that $B$ is torsion-free. Hence we can assume $W$ contains a vector $w$ with 3 non-zero coefficients.

 Thus, $w\in W$ is of the form $w= \pi_i^*+\pi_j^*+\pi_k^*$. Suppose $i,j$ both belong to the same set $\{1,2\}$, $\{3,4\}$ or $\{5,6\}$. 
Wlog $i=1,j=2, k=3$. Then we see that either $B$ is trivial contrary to assumption, or $\pi_3^*=-2\pi_1^*$ in $B$.
Now, the action of $G^0$ on the 3 element set $\{\{1,2\},\{3,4\},\{5,6\}\}$  is transitive, so it must contain both 3-cycles. Acting on the relation
$\pi_3^*=-2\pi_1^*$  in $B$ by $G^U$ thus gives $\pi_5^* = \pm 2\pi_3^*$ in $B$ and $\pi_1^*=\pm 2\pi_5^*$, and so $\pi_1^*=\pm 8\pi_1^*$ in $B$. Hence
$B$ is cyclic, generated by $\pi_1^*$, and $\pi_1^*$ has order $3,7$ or $9$.  But $v-w = \pi_4^*+\pi_5^*+\pi_6^*$ is in $U$, which gives the relation
$\pi_4^*=-2\pi_5^*=\pm \pi_1^*$ in $B$, so that $\pi_1^*=\pm 2\pi_1^*$ and $\pi_1^*$ must have order 3 in  $B$. We can now proceed as in case i. 

Hence we can assume that
any $w$ with 3 non-zero coefficients must be of the form $w = \pi_i^*+\pi_j^*+\pi_k^*$ with each of $\{i,j,k\}$ belonging to a different element
of $\{\{1,2\},\{3,4\},\{5,6\}\}$. Recalling the relations $\pi_1^*=\pm\pi_2^*,\pi_3^*=\pm\pi_4^*$ and $\pi_5^*=\pm\pi_6^*$ in $B$, we see that each such $w$
imposes a relation in $B$ of the form $\pi_1^*\pm\pi_3^*\pm\pi_5^* = 0$. Any 2 distinct such relations imply that $2\pi_i^*=0$ in $B$ for some $i$, which we
assumed was not the case. Hence $B$ is isomorphic as an abelian group to the free module on $\pi_1^*,\pi_3^*,\pi_5^*$ modded out by a single relation
$\pi_1^*\pm\pi_3^*\pm\pi_5^* = 0$. Hence $B$ is torsion free, which is a contradiction.

\end{itemize}

\end{proof}

\end{proof}

 


\end{document}